\documentclass[a4paper]{amsart}

\usepackage[T1]{fontenc}
\usepackage{amsmath, amssymb, amsthm, mathrsfs, mathtools, graphicx}
\usepackage{varioref}
\usepackage[backref=page]{hyperref}
\usepackage{enumitem, xspace, ifthen, comment}
\usepackage[all]{xy}
\usepackage[nameinlink, capitalize]{cleveref}
\usepackage[dvipsnames]{xcolor}
\usepackage{tikz}

\setlist[enumerate]{
  label=(\thethm.\arabic*),
  before={\setcounter{enumi}{\value{equation}}},
  after={\setcounter{equation}{\value{enumi}}},
  itemsep=1ex
}

\setlist[itemize]{
  leftmargin=*,
  topsep=1ex,
  itemsep=1ex,
  label=$\circ$
}

\newtheorem*{thm-plain}{Theorem}
\newtheorem{thm}{Theorem}[section]
\newtheorem{lem}[thm]{Lemma}
\newtheorem{prp}[thm]{Proposition}
\newtheorem{cor}[thm]{Corollary}

\newtheorem{conj}[thm]{Conjecture}
\newtheorem{ques}[thm]{Question}

\numberwithin{equation}{thm}

\theoremstyle{definition}
\newtheorem{dfn}[thm]{Definition}

\newtheorem*{dfn-plain}{Definition}

\theoremstyle{remark}
\newtheorem{clm}[thm]{Claim}

\newtheorem{awlog}[thm]{Additional Assumption}
\newtheorem{ntn}[thm]{Notation}

\newtheorem*{rem-plain}{Remark}

\newcommand{\inv}{^{-1}}
\newcommand{\from}{\colon}

\newcommand{\lto}{\longrightarrow}

\newcommand{\x}{\times}
\newcommand{\inj}{\hookrightarrow}

\newcommand{\bij}{\xrightarrow{\,\smash{\raisebox{-.5ex}{\ensuremath{\scriptstyle\sim}}}\,}}
\newcommand{\isom}{\cong}
\newcommand{\defn}{\coloneqq}

\newcommand{\tensor}{\otimes}

\newcommand{\wt}{\widetilde}

\renewcommand{\d}{\mathrm d}
\newcommand{\del}{\partial}
\newcommand{\delbar}{\overline\partial}

\newcommand{\dual}{^{\smash{\scalebox{.7}[1.4]{\rotatebox{90}{\textup\guilsinglleft}}}}}
\newcommand{\ddual}{^{\smash{\scalebox{.7}[1.4]{\rotatebox{90}{\textup\guilsinglleft} \hspace{-.5em} \rotatebox{90}{\textup\guilsinglleft}}}}}
\newcommand{\acts}{\ \rotatebox[origin=c]{-90}{\ensuremath{\circlearrowleft}}\ }

\newcommand{\factor}[2]{\left. \raise 2pt\hbox{$#1$} \right/\hskip -2pt \raise -2pt\hbox{$#2$}}

\newdir{ ir}{{}*!/-5pt/@^{(}} 
\newdir{ il}{{}*!/-5pt/@_{(}} 

\DeclareMathOperator{\re}{Re}

\newcommand{\set}[1]{\left\{ #1 \right\}}

\def\rd#1.{\lfloor{#1}\rfloor}
\def\rp#1.{\lceil{#1}\rceil}
\def\tw#1.{\langle{#1}\rangle}

\renewcommand{\O}[1]{\mathscr{O}_{#1}}
\newcommand{\Omegap}[2]{\Omega_{#1}^{#2}}
\newcommand{\Omegar}[2]{\Omega_{#1}^{[#2]}}
\newcommand{\T}[1]{\mathscr{T}_{#1}}
\newcommand{\can}[1]{\omega_{#1}}

\newcommand{\codim}[2]{\mathrm{codim}_{#1}(#2)}

\newcommand{\cc}[2]{\mathrm{c}_{#1}(#2)}

\newcommand{\PH}[1]{\mathrm{PH}_{#1,\R}}
\newcommand{\KK}[1]{\mathscr{K}^1_{#1,\R}}
\newcommand{\Cinfty}[1]{\mathscr{C}^\infty_{#1,\R}}

\def\Hnought#1.#2.{\mathit{\Gamma} \!\left( #1, #2 \right)}
\def\HH#1.#2.#3.{\mathrm{H}^{#1} \!\left( #2, #3 \right)}
\def\hh#1.#2.#3.{h^{#1} \!\left( #2, #3 \right)}
\def\RR#1.#2.#3.{R^{#1} #2_* #3}
\def\HHc#1.#2.#3.{\mathrm{H}_{\mathrm{c}}^{#1} \!\left( #2, #3 \right)}
\def\Hh#1.#2.#3.{\mathrm{H}_{#1} \!\left( #2, #3 \right)}
\def\Hom#1.#2.{\mathrm{Hom} \!\left( #1, #2 \right)}
\def\End#1.{\mathrm{End} \!\left( #1 \right)}
\def\sHom#1.#2.{\mathscr{H}\!om \!\left( #1, #2 \right)}
\def\Ext#1.#2.#3.{\mathrm{Ext}^{#1} \!\left( #2, #3 \right)}
\def\sExt#1.#2.#3.{\mathscr{E}\!xt^{#1} \!\left( #2, #3 \right)}

\newcommand{\kahler}{K{\"{a}}hler\xspace}
\newcommand{\cy}{Calabi--Yau\xspace}
\newcommand{\lt}{locally trivial\xspace}

\DeclareMathOperator{\Spec}{Spec}
\DeclareMathOperator{\divisor}{div}

\DeclareMathOperator{\Exc}{Exc}

\DeclareMathOperator{\Def}{Def}
\DeclareMathOperator{\Deflt}{Def^{lt}}
\DeclareMathOperator{\vol}{vol}

\newcommand{\eps}{\varepsilon}
\renewcommand{\theta}{\vartheta}
\renewcommand{\phi}{\varphi}

\newcommand{\N}{\ensuremath{\mathbb N}}
\newcommand{\Z}{\ensuremath{\mathbb Z}}
\newcommand{\Q}{\ensuremath{\mathbb Q}}
\newcommand{\R}{\ensuremath{\mathbb R}}
\newcommand{\C}{\ensuremath{\mathbb C}}

\renewcommand{\frm}{\mathfrak m}

 \newcommand{\frU}{\mathfrak U} 
\newcommand{\frX}{\mathfrak X} \newcommand{\frY}{\mathfrak Y}

  \newcommand{\sF}{\mathscr F}

  \newcommand{\cR}{\mathcal R}
  \newcommand{\cU}{\mathcal U}

\definecolor{forrest}{RGB}{81,133,49}
\definecolor{mydarkblue}{RGB}{10,92,153}

\title[The Kodaira problem for K{\"{A}}hler spaces with $\mathrm c_1 = 0$]{The Kodaira problem for K{\"{A}}hler spaces with vanishing first Chern class}

\author{Patrick Graf}
\address{Lehrstuhl f\"ur Mathematik I, Universit\"at Bayreuth, 95440 Bayreuth, Germany}
\email{\href{mailto:patrick.graf@uni-bayreuth.de}{patrick.graf@uni-bayreuth.de}}
\urladdr{\href{http://www.pgraf.uni-bayreuth.de/en/}{www.graficland.uni-bayreuth.de}}

\author{Martin Schwald}
\address{Fakult\"at f\"ur Mathematik, Universit\"at Duisburg--Essen, 45117 Essen, Germany}
\email{\href{mailto:martin.schwald@uni-due.de}{martin.schwald@uni-due.de}}
\urladdr{\href{http://www.esaga.uni-due.de/martin.schwald/}{www.esaga.uni-due.de/martin.schwald/}}

\date{May 6, 2021}
\thanks{The first author was partially supported by a DFG Research Fellowship.
The second author was partially supported by the DFG Collaborative Research Centre SFB/TR 45.}
\keywords{\kahler spaces, vanishing first Chern class, algebraic approximation, small projective deformations, locally trivial deformations}
\subjclass[2010]{32J27, 14E30, 32G05}

\hypersetup{
  pdfauthor={Patrick Graf, Martin Schwald},
  pdftitle={The Kodaira problem for Kahler spaces with vanishing first Chern class},
  pdfkeywords={Kahler spaces, vanishing first Chern class, algebraic approximation, small projective deformations, locally trivial deformations},
  pdfstartview={Fit},
  pdfpagelayout={OneColumn},
  pdfpagemode={UseNone},
  linktoc=all,
  breaklinks,
  linkcolor=[RGB]{0 0 96},
  citecolor=[RGB]{96 0 0},
  urlcolor=[RGB]{0 96 0},
  colorlinks}

\begin{document}

\begin{abstract}
Let $X$ be a normal compact \kahler space with klt singularities and torsion canonical bundle.
We show that $X$ admits arbitrarily small deformations that are projective varieties if its locally trivial deformation space is smooth.
We then prove that this unobstructedness assumption holds in at least three cases: if $X$ has toroidal singularities, if $X$ has finite quotient singularities, and if the cohomology group $\HH2.X.\T X.$ vanishes.
\end{abstract}

\maketitle

\begingroup
\hypersetup{linkcolor=black}
\tableofcontents
\endgroup

\section{Introduction}

\subsection{The Kodaira problem}

A compact complex manifold is projective if and only if it admits a \kahler form whose cohomology class is rational.
Thus it is a natural question to ask whether any compact \kahler manifold $X$ can be made projective by a (small) deformation of its complex structure.
This question is made precise by the notion of an \emph{algebraic approximation} of $X$, see \cref{def alg approx}.
The problem was already studied by Kodaira, who proved that every compact \kahler surface can be deformed to an algebraic surface~\cite{Kod63}, and it has been known as the \emph{Kodaira problem} ever since.

Recent progress gives a complete positive answer to the Kodaira problem in dimension three \cite{AlgApprox, ClaudonHoeringLinKahlerGroups, Lin17b}.
Also, examples of Voisin show that starting in dimension four there are compact \kahler manifolds without algebraic approximations \cite{Voi04,Voi06}.
Voisin's examples show how blowups can destroy the existence of algebraic approximations.
However, having the philosophy of the Minimal Model Program (MMP) in mind, this raises the question whether the Kodaira problem for $X$ should instead be posed for a (possibly singular) minimal model of $X$.
In this spirit, the following conjecture of Peternell becomes of interest:

\begin{conj}[Peternell] \label{peternell}
Let $X$ be a minimal compact \kahler space.
Then~$X$ admits an algebraic approximation.
\end{conj}

\noindent Related positive results in higher dimensions are still scarce \cite{Lin16}, and from the MMP point of view no natural extension of this conjecture to the case $\kappa(X) = -\infty$ is expected to hold \cite{MoriKodaira}.

\subsection{The case of Kodaira dimension zero}

When studying the Kodaira problem in higher dimensions, it is natural to aim at the case of Kodaira dimension zero first.
For a compact \kahler \emph{manifold} $X$ with trivial canonical bundle $K_X \sim 0$, the existence of an algebraic approximation can be proved by a combination of three well-known results:
\label{stepsABC}
\begin{itemize}[label=\textsf{(A)}]
\item[\textsf{(A)}] The BTT (= Bogomolov--Tian--Todorov) theorem states that the deformations of~$X$ are unobstructed, i.e.~the semiuniversal deformation space $\Def(X)$ is smooth~\cite{Bogomolov78, Tian87, Todorov89}.
\item[\textsf{(B)}] By the Green--Voisin criterion it is then sufficient to prove that cupping with a \kahler class $[\omega]\in\HH1.X.\Omegap X1.$ induces a surjection $\HH1.X.\T X. \lto \HH2.X.\O X.$ \cite[Prop.~5.20]{Voi03}.
\item[\textsf{(C)}] Using the assumption $K_X \sim 0$, the latter surjectivity is easily reduced to the Hard Lefschetz theorem.
\end{itemize}
More generally, if $X$ is a compact \kahler manifold with torsion canonical bundle $K_X \sim_\Q 0$, one passes to the index one cover $X_1 \to X$, on which the canonical bundle becomes trivial.
An algebraic approximation of $X_1$ then induces an algebraic approximation of $X$.
With Peternell's conjecture in mind, in the present work we study the case of a general minimal compact \kahler space $X$ of Kodaira dimension zero by extending the above line of reasoning.

\subsection{Main results}

If the conjectures of the \kahler MMP hold, studying \cref{peternell} for Kodaira dimension zero means considering normal compact \kahler spaces~$X$ with klt singularities such that the canonical class is torsion, $K_X \sim_\Q 0$.\footnote{
To be technically precise, we require that a reflexive tensor power of the canonical sheaf $\can X$ is trivial, i.e.~$\can X^{[m]} \isom \O X$ for some $m > 0$.}
As our first main result we establish that steps~\textsf{(B)} and~\textsf{(C)} above indeed generalize to this setting if one considers not all the deformations of a given $X$, but only the \emph{\lt} ones, see \cref{dfn lt def}.
This is done in Theorems~\labelcref{approx crit lt} and~\labelcref{hard lefschetz}, respectively.
Taken together, they yield the following.

\begin{thm}[Algebraic approximation for unobstructed spaces] \label{main unobstructed}
Let $X$ be a normal compact \kahler space with klt singularities and $K_X \sim_\Q 0$.
If the base of the semi\-universal \lt deformation $\frX \to \Deflt(X)$ is smooth, then $\frX \to \Deflt(X)$ is a strong algebraic approximation of $X$.
\end{thm}

In light of this result, it becomes important to study step~\textsf{(A)}, that is, to prove a BTT theorem in our setting.
The problem is that while there do exist some singular versions of BTT for threefolds and under other strong hypotheses, it is also well-known that the full result becomes invalid in the presence of singularities: examples by Gross show that for $n \ge 3$ the deformation space of a \cy $n$-fold with canonical singularities need not even be locally irreducible~\cite{Gro97}.
Gross' examples, however, are easily seen not to be \lt.
In particular, in the literature there is no answer to the following question:

\begin{ques}[Locally trivial BTT] \label{BTT question}
Is $\Deflt(X)$ smooth for any normal compact \kahler space with klt singularities and $K_X \sim_\Q 0$?
\end{ques}

For primitive symplectic varieties, the answer to this question is \emph{yes} due to Bakker and Lehn~\cite[Thm.~4.11]{BakkerLehn18}.
We do not know whether the answer is always \emph{yes}, but we do prove that \cref{BTT question} has a positive solution in the following two cases.
The first one is a generalization of the well-known unobstructedness criterion of Kodaira--Nirenberg--Spencer~\cite{KodairaNirenbergSpencer58} to \lt deformations.

\begin{thm}[Locally trivial KNS theorem] \label{main 2}
Let $X$ be a normal compact complex space with $\HH2.X.\T X. = 0$.
Then $\Deflt(X)$ is smooth.
\end{thm}

\begin{rem-plain}
A stronger statement than \cref{main 2} is that $\HH2.X.\T X.$ is an obstruction space for the \lt deformation functor $D_X^{\mathrm{lt}}$ of $X$.
This means that for every sequence $0 \to J \to A' \to A \to 0$ of local Artinian \C-algebras with residue field \C\ and such that $\frm_{A'} J = 0$, there is an exact sequence of pointed sets $D_X^{\mathrm{lt}}(A') \to D_X^{\mathrm{lt}}(A) \to \HH2.X.\T X. \tensor J$.
This is claimed (in the algebraic case) in~\cite[Prop.~2.4.6]{Ser06}, but the proof contains a gap: in order to extend a local automorphism of a deformation of~$X$ over $A$ to $A'$, the Infinitesimal Lifting Property is invoked.
This, however, is only possible if $X$ is smooth.
Our \cref{main 2} weakly fills this gap.
\end{rem-plain}

The second case where we can prove unobstructedness is when the Hodge theory of reflexive differential forms on $X$ behaves nicely.
More precisely, denote by $\Omegar X\bullet = j_* \Omegap U\bullet$ the complex of reflexive differential forms on $X$, where $j \from U \to X$ is the inclusion of the smooth locus.
Then we require two conditions: firstly, the natural map $\underline{\C}_X \to \Omegar X\bullet$ should be a quasi-isomorphism.
Equivalently, the Poincar\'e lemma for reflexive differential forms should hold on $X$: every closed reflexive form of degree $\ge 1$ is locally exact.
Under this assumption there is a Hodge to de Rham spectral sequence, coming from the \emph{filtration b\^ete} on $\Omegar X\bullet$:
\begin{equation} \label{hdr}
E_1^{p,q} = \HH q.X.{\Omegar Xp}. \Longrightarrow \HH{p+q}.X.\C..
\end{equation}
The second requirement is that this spectral sequence degenerate on the first page in degree~$n = \dim X$.
That is, we require the differentials $d_r^{p,q}$ to vanish for all $r \ge 1$ and $p + q = n$.

\begin{thm} \label{main 3}
Let $X$ be an $n$-dimensional normal compact \kahler space with rational singularities and trivial canonical class $K_X \sim 0$ such that the reflexive Poincar\'e lemma holds on $X$ and~\labelcref{hdr} degenerates at $E_1$ in degree~$n$.
Then $\Deflt(X)$ is smooth.
\end{thm}

The assumptions in \cref{main 3} are very strong: even for klt singularities, the reflexive Poincar\'e lemma can fail in degree at least two and if it holds, the spectral sequence~\labelcref{hdr} need not degenerate on the first page \cite[Prop.~1.6(b)]{Joerder}.
However, the assumptions do hold in at least two interesting cases:
\begin{itemize}
\item if $X$ has only (finite) quotient singularities~\cite{Steenbrink77}, and
\item if $X$ is a \emph{toroidal} complex space, i.e.~it is locally analytically isomorphic to the germ of a toric variety at a torus-fixed point~\cite{Danilov91}.
\end{itemize}
We deduce the following corollary.

\begin{cor}[Locally trivial BTT for orbifold/toroidal singularities] \label{cor 4}
Let $X$ be a normal compact \kahler space with either only quotient or only toroidal singularities and $K_X \sim_\Q 0$.
Then $\Deflt(X)$ is smooth.
\end{cor}

\noindent
Summing up Theorems~\labelcref{main unobstructed},~\labelcref{main 2} and \cref{cor 4}, we have the following positive partial answer to Peternell's conjecture:

\begin{cor} \label{cor 5}
Let $X$ be a normal compact \kahler space with $K_X \sim_\Q 0$.
Assume any of the following:
\begin{enumerate}
\item\label{cor 5.1} $X$ has only quotient singularities.
\item\label{cor 5.2} $X$ has only toroidal singularities.
\item\label{cor 5.3} $X$ is klt and the cohomology group $\HH2.X.\T X.$ vanishes.
\end{enumerate}
Then $X$ admits a strong algebraic approximation.
\end{cor}

\begin{rem-plain}
In the orbifold case~\labelcref{cor 5.1}, the analog of the Beauville--Bogomolov decomposition for~$X$ is already known~\cite[Thm.~6.4]{Cam04FanoConference}.
It is tempting to try and use this splitting result in order to obtain an algebraic approximation.
Our approach is completely different, as it does not rely on any kind of structure theorem.
\end{rem-plain}

\subsection{Recent progress}

Since this paper was first posted on the arXiv, substantial progress has been made on the questions discussed herein.
Building on our Theorems~\labelcref{approx crit lt} and~\labelcref{hard lefschetz} and on further work of the first-named author~\cite{BochnerCGGN}, \cref{peternell} and the Beauville--Bogomolov decomposition have been established for klt \kahler spaces with $\mathrm c_1 = 0$~\cite{BakkerGuenanciaLehn20, KodairaFour}.
This was done by proving a weak version~of the BTT unobstructedness theorem which only deals with deformations in the ``symplectic direction''.
In general, the answer to \cref{BTT question} remains unknown and our \cref{cor 4} remains one of the few partial results available.

\subsection*{Acknowledgements}

The first-named author would like to thank Professor E.~Sernesi for helpful discussions about deformation theory, in particular about~\cref{main 2}.
Both authors thank the anonymous referee for careful proofreading and several suggestions improving our exposition.

\section{Basic facts and definitions}

\subsection*{Complex spaces}

All complex spaces are assumed to be separated and connected, unless otherwise stated.
An open subspace of a complex space is called \emph{big} if its complement is analytic and has codimension at least two.

\subsection*{Deformation theory} \label{sec def theory}

We collect some basic definitions and facts from deformation theory.

\begin{dfn}[Deformations of complex spaces]
A \emph{deformation} of a reduced compact complex space $X$ is a proper flat morphism $\pi \from \frX \to (S, 0)$ from a (not necessarily reduced) complex space $\frX$ to a complex space germ $(S, 0)$, equipped with a fixed isomorphism $X_0 \isom X$, where we write $X_s \defn \pi\inv(s)$ for the fibre over any $s \in S$.
We usually suppress both the base point $0 \in S$ and the choice of isomorphism from notation.

If $S = \Spec A$ for some local Artinian \C-algebra $A$ with residue field \C, the deformation $\frX \to S$ is called \emph{infinitesimal}.
\end{dfn}

\begin{dfn}[Algebraic approximations] \label{def alg approx}
Let $X$ be a compact complex space and $\pi \from \frX \to S$ a deformation of $X$.
Consider the set of projective fibres
\[ S^{\mathrm{alg}} \defn \big\{ s \in S \;\big|\; \frX_s \text{ is projective} \big\} \subset S \]
and its closure $\overline{S^{\mathrm{alg}}} \subset S$.
We say that $\frX \to S$ is an \emph{algebraic approximation of $X$} if $0 \in \overline{S^{\mathrm{alg}}}$.
We say that $\frX \to S$ is a \emph{strong algebraic approximation of $X$} if $\overline{S^{\mathrm{alg}}} = S$ as germs, i.e.~$S^{\mathrm{alg}}$ is dense near $0 \in S$.
\end{dfn}

\begin{dfn}[Locally trivial deformations] \label{dfn lt def}
A deformation $\pi \from \frX \to S$ is called \emph{\lt} if for every $x \in \frX_0$ there exist open subsets $0 \in S^\circ \subset S$ and $x \in U \subset \pi\inv(S^\circ)$ and an isomorphism
\[ \xymatrix{
U \ar^-\sim[rr] \ar_-\pi[dr] & & (\frX_0 \cap U) \x S^\circ \ar^-{\operatorname{pr}_2}[dl] \\
& S^\circ. &
} \]
\end{dfn}

\noindent
Any compact complex space $X$ admits a semiuniversal \lt deformation $\frX \to \Deflt(X)$, and $\Deflt(X)$ is a closed subspace of the semiuniversal deformation space $\Def(X)$.
This is proved in~\cite[Cor.~0.3]{FK87}.

\subsection*{Resolution of singularities}

A \emph{resolution of singularities} of a complex space $X$ is a proper bimeromorphic morphism $f \from \wt X \to X$, where $\wt X$ is a complex manifold.
We say that the resolution is \emph{projective} if $f$ is a projective morphism.
In this case, if $X$ is projective (resp.~compact \kahler) then so is $\wt X$.
A resolution is said to be \emph{strong} if it is a biholomorphism over the smooth locus of $X$.

\begin{thm}[Functorial resolutions] \label{funct res}
There exists a \emph{resolution functor} which assigns to any complex space $X$ a strong projective resolution $\pi_X \from \cR(X) \to X$, such that $\cR$ commutes with smooth maps in the following sense:
For any smooth morphism $f \from W \to X$, there is a unique smooth morphism $\cR(f) \from \cR(W) \to \cR(X)$ such that the following diagram is a fibre product square.
\[ \xymatrix{
\cR(W) \ar^-{\cR(f)}[rr] \ar_-{\pi_W}[d] & & \cR(X) \ar^-{\pi_X}[d] \\
W \ar^-f[rr] & & X.
} \]
\end{thm}

\begin{proof}
See~\cite[Thm.~3.45]{Kol07}.
\end{proof}

\begin{dfn}[Simultaneous resolutions]
Let $\frX \to S$ be a deformation of $X$.
A \emph{simultaneous resolution} of $\frX \to S$ is a proper bimeromorphic map $\frY \to \frX$ such that the composition $\frY \to S$ is a smooth morphism.
\end{dfn}

Unlike the general case, simultaneous resolutions do exist for \lt deformations.
This is proved in~\cite[Lemma~4.6]{BakkerLehn18}, but see also~\cite[Lemma~4.2]{MoriKodaira}.
If the deformation in question is locally given as a product $U \x S$ with $U \subset X$ open, then on the preimage sets, $\frY$ will be given by $\cR(U) \x S$.
In particular, the restriction of $\frY \to \frX$ to the central fibre is the functorial resolution $Y = \cR(X) \to X$, and $\frY \to S$ is a (\lt) deformation of $Y$.

\subsection*{Reflexive sheaves on non-reduced spaces}

For the rest of this section, we fix the following setting: $X$ is a normal compact complex space and $\pi \from \frX \to S = \Spec A$ is an infinitesimal deformation of $X$.

\begin{lem} \label{531}
The sheaf $\O\frX$ is torsion-free and normal, i.e.~for any open subset $U \subset \frX$ and any analytic subset $Z \subset U$, the following holds:
\begin{enumerate}
\item\label{531.1} The restriction map $\Hnought U.\O\frX. \to \Hnought U \setminus Z.\O\frX.$ is injective.
\item\label{531.2} If $\codim \frX Z \ge 2$, then the above restriction map is bijective.
\end{enumerate}
\end{lem}

\begin{proof}
Let $\frm \subset A$ be the maximal ideal, and set $S_n = \Spec A_n$, where $A_n \defn \factor A{\frm^{n+1}}$.
We will prove by induction on $n$ that the claims hold for the base change $\frX_n \to S_n$, for every $n \in \N$.
This is sufficient because $A$ is Artinian and so $A_n = A$ for $n \gg 0$.

For $n = 0$, the statement is true because $\frX_0 = X$ is assumed to be normal.
For the inductive step, assume the statement for $\frX_{n-1}$ and consider the short exact sequence
\begin{equation} \label{ses artin}
0 \lto \factor{\frm^n}{\frm^{n+1}} \lto \factor A{\frm^{n+1}} \lto \factor A{\frm^n} \lto 0.
\end{equation}
Note that $\factor{\frm^n}{\frm^{n+1}}$ is annihilated by $\frm$, hence it is in fact an $\factor A\frm$-module (that is, a \C-vector space, say of dimension $N$) and its $A$-module structure is obtained by pullback along the natural map $A \to \C$.
Thus, pulling back sequence~\labelcref{ses artin} to $\frX$, by flatness we obtain an exact sequence
\[ 0 \lto \O X^{\oplus N} \lto \O{\frX_n} \lto \O{\frX_{n-1}} \lto 0. \]
Torsion-freeness~\labelcref{531.1} follows easily from this by a diagram chase.
For normality~\labelcref{531.2}, note that extendability over $Z$ is a local property because of torsion-freeness.
We may thus assume that $\frX$ is Stein.
In this case, the claim again follows from a diagram chase.
\end{proof}

It follows from \cref{531} that $\sF\dual \defn \sHom\sF.\O\frX.$ is torsion-free and normal for any coherent sheaf $\sF$ on $\frX$.
In particular, if $\sF$ is reflexive in the sense that the natural map $\sF \to \sF\ddual$ to the double dual is an isomorphism, then $\sF$ is torsion-free and normal.
Conversely, if $\sF$ is torsion-free and normal and locally free on a big open subset $\frU \subset \frX$, then it is reflexive.

Now let $\sF$ be coherent and locally free on a big open subset $\frU \subset \frX$ (but not necessarily torsion-free and normal).
Then the double dual of the natural map $\sF \to i_* i^* \sF$ is an isomorphism, where $i \from \frU \inj \frX$ denotes the inclusion.
In other words, $\sF\ddual \isom i_* i^* \sF$ for such a sheaf.

We apply the above observations to various sheaves of \kahler differentials.
Let $\Omegap{\frX/S}1$ be the sheaf of relative differentials of the morphism $\pi$, and set $\Omegap{\frX/S}p \defn \bigwedge^p \Omegap{\frX/S}1$ for any integer $p \ge 0$.
The smooth locus $i \from \frU \inj \frX$ of $\pi$ is a big open subset, and $i^* \Omegap{\frX/S}1 = \Omegap{\frU/S}1$ is locally free.
We define the sheaf of \emph{relative reflexive $p$-forms} to be $\Omegar{\frX/S}p \defn \big( \Omegap{\frX/S}p \big) \ddual = i_* \big( \Omegap{\frU/S}p \big)$.
For $p = \dim X$, we obtain the \emph{relative canonical sheaf} $\can{\frX/S}$.
In a similar vein, we have the \emph{relative tangent sheaf} $\T{\frX/S} = i_* \big( \T{\frU/S} \big)$.
If the deformation $\frX \to S$ is \lt, then after shrinking $\frX$ we may write $\frX = X \x S$.
In this case, $\Omegar{\frX/S}p = \mathrm{pr}_1^* \, \Omegar Xp$, where $\mathrm{pr}_1$ is the projection onto the first factor.
In particular the sheaf $\Omegar{\frX/S}p$ is flat over $S$, and its restriction to the central fibre is $\Omegar Xp$.

Continuing to assume local triviality, let $g \from \frY \to \frX$ be a simultaneous resolution of $\frX \to S$.
Then there is a natural push-forward map $g_* \Omegap{\frY/S}p \inj \Omegar{\frX/S}p$.

\begin{prp} \label{gkkp}
If the central fibre $X$ has rational singularities, then the above map $g_* \Omegap{\frY/S}p \inj \Omegar{\frX/S}p$ is an isomorphism for any $p \ge 0$.
\end{prp}

\begin{proof}
We employ induction on $n$, using the same notation and line of reasoning as in the proof of \cref{531}.
Let $\bar g \from Y \to X$ be the restriction of $g$ to the central fibre.
For the start of induction, we need to know that $\bar g_* \Omegap Yp = \Omegar Xp$, which follows from~\cite[Cor.~1.8]{KebekusSchnell18} as $X$ has rational singularities.
For the inductive step, tensor sequence~\labelcref{ses artin} with the map in question.
By local triviality, we obtain a commutative diagram
\[ 
\xymatrix{
0 \ar[r] & g_* \Omegap{\frY/S}p \tensor \O X^{\oplus N} \ar[r] \ar^-{\beta_n}[d] & g_* \Omegap{\frY/S}p \tensor \O{\frX_n} \ar[r] \ar[d]^-{\alpha_n} & g_* \Omegap{\frY/S}p \tensor \O{\frX_{n-1}} \ar[r] \ar^-{\alpha_{n-1}}[d] & 0 \\
0 \ar[r] & \Omegar{\frX/S}p \tensor \O X^{\oplus N} \ar[r] & \Omegar{\frX/S}p \tensor \O{\frX_n} \ar[r] & \Omegar{\frX/S}p \tensor \O{\frX_{n-1}} \ar[r] & 0.
} \]
The map $\beta_n = \alpha_0^{\oplus N}$ is just a finite direct sum of copies of $\alpha_0$.
In particular, it is an isomorphism.
Likewise, $\alpha_{n-1}$ is an isomorphism by the inductive assumption.
It now follows from the Five Lemma that also $\alpha_n$ is an isomorphism.
\end{proof}

\subsection*{\kahler metrics on singular spaces}

In~\cite{TorusQuotients}, the first author associated to a \kahler metric on a singular complex space $X$ a cohomology class in the topological cohomology group $\HH2.X.\R.$.
Here, our focus will instead be on the (coherent) sheaf of \kahler differentials.
The content of this section is probably well-known to experts.
Unfortunately, we have not been able to locate it in the published literature.
The interested reader may check that all of the following is compatible with the usual notions if $X$ is smooth.

\begin{ntn}[Pluriharmonic functions, \protect{\cite[pp.~17,~23]{Var89}}] \label{PH X}
Let $X$ be a reduced complex space.
We denote by $\Cinfty X$ the sheaf of smooth real-valued functions on $X$.
Moreover, we denote by $\PH X$ the image of the real part map $\re \from \O X \to \Cinfty X$, which is called the sheaf of real-valued \emph{pluriharmonic functions} on $X$, and we set $\KK X \defn \factor{\Cinfty X}{\PH X}$.
\end{ntn}

The short exact sequence defining the sheaf $\KK X$ gives rise to a connecting homomorphism
\[ \delta \from \KK X(X) \lto \HH1.X.\PH X.. \]
Furthermore, since a holomorphic function with vanishing real part is locally constant, we have an isomorphism
\[ \PH X \isom \factor{\O X}{\underline\R_X} \]
where $\underline\R_X \inj \O X$ is embedded via multiplication by $\mathrm i$.
Consequently we get an exterior differentiation map $\d \from \PH X \to \Omegap X1$, where $\Omegap X1$ is the sheaf of (universally finite) \kahler differentials~\cite[Satz~1.2]{GK64}.
This induces a map on cohomology
\[ \d \from \HH1.X.\PH X. \lto \HH1.X.\Omegap X1.. \]
Composing, for any element $\kappa \in \KK X(X)$ we get a class
\[ [\kappa] \defn (\d \circ \delta)(\kappa) \in \HH1.X.\Omegap X1.. \]

\begin{dfn}[\kahler metrics, \protect{\cite[pp.~23,~18]{Var89}}] \label{kahler def}
Let $X$ be a reduced complex space.
\begin{enumerate}
\item A \emph{\kahler metric} on $X$ is an element $\omega$ of $\KK X(X)$ which can be represented by a family $(U_i,\phi_i)_{i \in I}$ such that $\phi_i$ is a smooth strictly plurisubharmonic function on $U_i$ for all $i \in I$.
That is, locally $\phi_i$ is the restriction of a smooth strictly plurisubharmonic function on an open subset of $\C^{N_i}$ under a local embedding $U_i \inj \C^{N_i}$.
\item We say that $c \in \HH1.X.\Omegap X1.$ is a \emph{\kahler class} on $X$ if there exists a \kahler metric $\omega$ on $X$ such that $c = [\omega]$.
\item We say that $X$ is \emph{\kahler} if there exists a \kahler metric on $X$.
\end{enumerate}
\end{dfn}

Of course, a \kahler class $[\omega]$ can be mapped further to $\HH1.X.\Omegar X1.$, where $\Omegar X1$ is the sheaf of reflexive differentials.
Knowing that the class comes from the \kahler differentials however provides useful additional information: it shows that the class can be pulled back along arbitrary morphisms, even in situations where the extension theorem of~\cite{KebekusSchnell18} cannot be applied.

\section{A criterion for algebraic approximability}

The following theorem generalizes the Green--Voisin criterion for algebraic approximability mentioned as step~\textsf{(B)} in the introduction \vpageref{stepsABC}.
Let $X$ be a reduced complex space.
To any class $c \in \HH1.X.\Omegap X1.$, one can associate a linear map
\[ \alpha_c \from \HH1.X.\T X. \xrightarrow{c \cup -} \HH2.X.\Omegap X1 \tensor \T X. \xrightarrow{\text{contraction}} \HH2.X.\O X.. \]

\begin{thm}[Approximation criterion for \lt deformations] \label{approx crit lt}
Let $X$ be a compact \kahler space with rational singularities, and let $\pi \from \frX \to S$ be a locally trivial deformation of $X$ over a \emph{smooth} base $S$.
Denote by
\[ \kappa(\pi) \from T_0 S \lto \HH1.X.\T X. \]
the Kodaira--Spencer map of $\pi$.
Assume that for some \kahler class $[\omega] \in \HH1.X.\Omegap X1.$, the composition
\[ \alpha_{[\omega]} \circ \kappa(\pi) \from T_0 S \lto \HH2.X.\O X. \]
is surjective.
Then $\pi$ is a strong algebraic approximation of $X$.
\end{thm}

Recall that the Kodaira--Spencer map $\kappa(\pi)$ for a \lt deformation of a singular space can be defined much in the same way as in the smooth case, cf.~e.g.~\cite[Sec.~9.1.2]{Voi02}.
In particular, this implies that the left-hand square in the commutative diagram below actually is commutative.

\begin{proof}
Let $g \from \frY \to \frX$ be a simultaneous resolution of $\frX \to S$ with $\bar g \from Y \to X$ the restriction to the central fibre, and let $f \defn \pi \circ g \from \frY \to S$ be the composition.
Furthermore, set $[\wt\omega] \defn [\bar g^* \omega]$.
Then there is the following commutative diagram:
\[ \xymatrix{
T_0 S \ar@{=}[d] \ar^-{\kappa(f)}[rr] & & \HH1.Y.\T Y. \ar^-{\alpha_{[\wt\omega]}}[rr] & & \HH2.Y.\O Y. \ar@{=}[d] \\
T_0 S \ar^-{\kappa(\pi)}[rr] & & \HH1.X.\T X. \ar@{ ir->}[u] \ar^-{\alpha_{[\omega]}}[rr] & & \HH2.X.\O X.
} \]
The vertical map on the left is the identity.
The map in the middle comes from the fact that $\bar g_* \T Y = \T X$, which holds because $\bar g$ is the functorial resolution of~$X$ (see~\cite[Thm.~4.2]{GK13}).
The vertical map on the right is an isomorphism since $X$ has rational singularities.
Thus we see that
\[ \alpha_{[\wt\omega]} \circ \kappa(f) \from T_0 S \lto \HH2.Y.\O Y. \]
is still surjective.
Of course, $[\wt\omega]$ is not \kahler unless $X$ is already smooth.
But as~$g$ (and hence also $\bar g$) is a projective morphism, there exists a $\bar g$-ample line bundle $L$ on $Y$.
For $\eps > 0$, consider the class $[\wt\omega_\eps] \defn [\wt\omega] + \eps \cdot \cc1L$ and observe that $\alpha_{[\wt\omega_\eps]} = \alpha_{[\wt\omega]} + \eps \cdot \alpha_{\cc1L}$.
Hence we can choose $\eps$ sufficiently small such that the map $\alpha_{[\wt\omega_\eps]} \circ \kappa(f)$ remains surjective and the class $[\wt\omega_\eps]$ is \kahler~\cite[Prop.~II.1.3.1(vi)]{Var89}.
Thus $f$ is a strong algebraic approximation of~$Y$ by the Green--Voisin criterion~\cite[Prop.~5.20]{Voi03}.

This implies that $\pi$ is a strong algebraic approximation of $X$:
As we assumed $X$ to be a compact \kahler space with rational singularities, we may shrink $S$ around its base point such that for every $t\in S$ the fibre $X_t$ is still \kahler~\cite[Thm.~6.3]{Bin83} with rational singularities (by local triviality of $\pi$, or by~\cite[Th\'eor\`eme~4]{Elkik78}).
If the fibre $Y_t$ for some $t \in S$ is projective, then $X_t$ is Moishezon and we conclude from~\cite[Thm.~1.6]{Nam02} that $X_t$ is also projective.
\end{proof}

\section{Surjectivity of cupping with a \kahler class}

In this section, we generalize step~\textsf{(C)} from the introduction \vpageref{stepsABC} to our singular setting.
We then prove \cref{main unobstructed}.

\begin{thm} \label{hard lefschetz}
Let $X$ be a normal compact complex space with klt singularities and $K_X \sim_\Q 0$.
Then for any \kahler class $[\omega] \in \HH1.X.\Omegap X1.$, the map $\alpha_{[\omega]}$ is surjective.
\end{thm}

\begin{proof}
Let $\pi \from X_1 \to X$ be the index one cover of $X$.
Then there is a finite (cyclic) group $G$ acting on $X_1$ such that $\pi$ is the quotient map.
Since $K_X \sim_\Q 0$, the map $\pi$ is quasi-\'etale and hence the $G$-action is free in codimension one.
Furthermore $X_1$ has canonical (in particular, rational) singularities and $K_{X_1}$ is trivial~\cite[Prop.~5.20]{KM98}.
The pullback $[\pi^* \omega]$ is a $G$-invariant \kahler class on $X_1$ by~\cite[Prop.~3.5]{TorusQuotients}.
Assume for the moment that we can prove surjectivity of the $G$-equivariant map $\alpha_{[\pi^* \omega]}$.
Then, as taking $G$-invariants is an exact functor~\cite{Mas99}, also the restriction to the invariant subspaces is surjective.
By~\cite[Lemma~5.3]{AlgApprox}, we obtain a commutative diagram
\[ \xymatrix{
\HH1.X_1.\T{X_1}. \ar^-{\alpha_{[\pi^* \omega]}}@{->>}[rr] & & \HH2.X_1.\O{X_1}. \\
\HH1.X.\T X. \ar^-{\alpha_{[\omega]}}@{->>}[rr] \ar@{ ir->}[u] & & \HH2.X.\O X. \ar@{ ir->}[u] \\
} \]
Thus we may replace $X$ by $X_1$ and make the following additional assumption.

\begin{awlog}
The canonical sheaf of $X$ is trivial, $\can X \isom \O X$.
\end{awlog}

Fix a global generator $\sigma \in \HH0.X.\can X.$.
Multiplication by $\sigma$ gives an isomorphism $\O X \bij \can X$, and contracting $\sigma$ by vector fields similarly yields $\T X \bij \Omegar X{n-1}$.
We obtain the following diagram (whose vertical maps depend on the choice of $\sigma$):
\[ \xymatrix{
\HH1.X.\Omegar X{n-1}. \ar^-{[\omega] \cup -}[rr] & & \HH2.X.\can X. \\
\HH1.X.\T X. \ar^-{\alpha_{[\omega]}}[rr] \ar_-{\rotatebox{90}{$\sim$}}[u] & & \HH2.X.\O X. \ar_-{\rotatebox{90}{$\sim$}}[u]
} \]

\begin{clm} \label{597}
The above diagram is commutative.
\end{clm}

\begin{proof}[Proof of \cref{597}]
We will compute in \v Cech cohomology with respect to a fixed Leray cover $\cU \defn \set{ U_i \mid i \in I}$ of $X$.
Represent the given \kahler class $[\omega]$ by a cocycle $(\omega_{ij})$ with $\omega_{ij} \in \Hnought U_{ij}.\Omega_X^1.$, and pick an arbitrary class $[(v_{ij})] \in \HH1.X.\T X.$.
Recall that quite generally, the cup product of two $1$-cocycles $(\alpha_{ij})$ and $(\beta_{ij})$ is the $2$-cocycle given by $(\alpha_{ij} \tensor \beta_{jk})$.
Hence, if we first move upwards and then to the right, we obtain the cocycle $\big( \omega_{ij} \wedge \iota_{v_{jk}}(\sigma) \big)$, where $\iota$ denotes contraction.
If we go the other way round, we end up at $\big( \iota_{v_{jk}}(\omega_{ij}) \cdot \sigma \big)$.
These two cocycles are cohomologous.
In fact, they are equal on each open set $U_{ijk}$ because $\iota$ is an antiderivation and hence
\[ 0 = \iota_{v_{jk}} \big( \underbrace{\omega_{ij} \wedge \sigma}_{=0} \big) = \iota_{v_{jk}}(\omega_{ij}) \wedge \sigma - \omega_{ij} \wedge \iota_{v_{jk}}(\sigma). \qedhere \]
\end{proof}

By \cref{597}, it suffices to show that $\HH1.X.\Omegar X{n-1}. \xrightarrow{[\omega] \cup -} \HH2.X.\can X.$ is surjective.
To this end, consider a strong projective resolution $f \from Y \to X$ and the commutative diagram
\begin{equation} \label{cd hl}
\xymatrix{
\HH0.Y.\Omegap Y{n-2}. \ar^-{[f^*\omega] \cup -}[rr] \ar@{ ir->}[d] & & \HH1.Y.\Omegap Y{n-1}. \ar^-{[f^*\omega] \cup -}[rr] & & \HH2.Y.\can Y. \\
\HH0.X.\Omegar X{n-2}. \ar^-{[\omega] \cup -}[rr] & & \HH1.X.\Omegar X{n-1}. \ar[rr]^-{[\omega] \cup -} & & \HH2.X.\can X.. \ar@{=}[u] 
}
\end{equation}

\begin{clm} \label{cd hl 1}
The vertical maps in the above diagram are isomorphisms.
\end{clm}

\begin{proof}
Since $X$ has rational singularities, we have $f_* \can Y = \can X$ and the vertical map on the right is the natural map $\HH2.X.f_* \can Y. \to \HH2.Y.\can Y.$.
This is an isomorphism since $\RR i.f.\can Y. = 0$ for $i > 0$ by Grauert--Riemenschneider vanishing~\cite[Thm.~1]{Takegoshi85}.
The vertical map on the left is the pushforward map for reflexive differentials.
It is an isomorphism by~\cite[Cor.~1.8]{KebekusSchnell18}, although we actually do not use this fact.
\end{proof}

\begin{clm} \label{cd hl 2}
The upper horizontal map $\HH0.Y.\Omegap Y{n-2}. \xrightarrow{[f^*\omega]^2 \cup -} \HH2.Y.\can Y.$ in the above diagram is surjective.
\end{clm}

\begin{proof}
By the Hard Lefschetz theorem on the compact \kahler manifold $Y$ (or by the Hodge $*$-operator), the vector spaces $\HH0.Y.\Omegap Y{n-2}.$ and $\HH2.Y.\can Y.$ have the same dimension.
It is therefore sufficient to prove injectivity of the map in question.
So let $\alpha \in \HH0.Y.\Omegap Y{n-2}. \setminus \set0$ be a nonzero holomorphic $(n-2)$-form.
We claim that
\begin{equation} \label{int}
\underbrace{\mathrm i^{n-2} (-1)^{\frac{(n-2)(n-3)}2}}_{= \mathrm i^{n^2}} \int_Y \alpha \wedge \overline\alpha \wedge f^* \omega^2 > 0.
\end{equation}
Indeed, set $E = \Exc(f)$ and fix a volume form $\vol$ on $Y$ which on $Y \setminus E$ is compatible with the orientation induced by $f^* \omega$.
Then outside of $E_\alpha \defn E \cup \set{ \alpha = 0}$, the integrand in~\labelcref{int} is a strictly positive multiple of $\vol$ by the (pointwise) Hodge--Riemann bilinear relations~\cite[Cor.~1.2.36]{Huy05}.
Being an analytic subset of lower dimension, $E_\alpha$ has measure zero.
Inequality \labelcref{int} follows.

By Stokes' theorem and~\labelcref{int}, the form $\alpha \wedge \overline\alpha \wedge f^* \omega^2$ cannot be $\d$-exact.
On the other hand, the global holomorphic form $\alpha$ is $\d$-closed, so $\d\overline\alpha = \overline{\d\alpha} = 0$.
Thus also $\alpha \wedge f^* \omega^2$ is not $\d$-exact.
Equivalently, by the $\del\delbar$-lemma, $\alpha \wedge f^* \omega^2$ is not $\delbar$-exact.
Interpreting $\HH2.Y.\can Y.$ as a Dolbeault cohomology group, the latter form represents the class $[f^* \omega]^2 \cup \alpha$, which is therefore nonzero.
\end{proof}

The desired surjectivity of the lower right map in diagram~\labelcref{cd hl} now follows from \cref{cd hl 2} and a simple diagram chase, ending the proof of \cref{hard lefschetz}.
\end{proof}

\begin{rem-plain}
\cref{hard lefschetz} continues to hold for $X$ with rational singularities as long as $\can X$ is trivial.
However, taking the index one cover in general does not preserve rationality and then the proof has some problems.
An easy example is given by the (rational) cone over an Enriques surface, which has a quasi-\'etale cover by the (non-rational) cone over the corresponding K3 surface.
\end{rem-plain}

\subsection*{Proof of \cref{main unobstructed}}

Let $X$ be a normal compact \kahler space with klt singularities and $K_X \sim_\Q 0$.
We assume that the base of the semi\-universal \lt deformation $\pi \from \frX \to \Deflt(X)$ is smooth.
Let $[\omega] \in \HH1.X.\Omegap X1.$ be a \kahler class on~$X$.
By \cref{hard lefschetz}, cupping with $[\omega]$ induces a surjective map $\alpha_{[\omega]} \from \HH1.X.\T X. \lto \HH2.X.\O X.$.
By semiuniversality of $\pi$, the Kodaira--Spencer map $\kappa(\pi) \from T_0 \Deflt(X) \lto \HH1.X.\T X.$ is an isomorphism, and the composition
\[ \alpha_{[\omega]} \circ \kappa(\pi) \from T_0 \Deflt(X) \lto \HH2.X.\O X. \]
is surjective.
By \cref{approx crit lt}, it follows that $\pi$ is a strong algebraic approximation of~$X$. \qed

\section{Locally trivial unobstructedness}

\stepcounter{thm}

\subsection{$T^1$-lifting principle for locally trivial deformations}

Let $X$ be a normal compact complex space.
We want to prove smoothness of $\Deflt(X)$ under the assumptions of Theorems~\labelcref{main 2} and~\labelcref{main 3}, respectively.
In both cases, our aim is to apply Kawamata--Ran's $T^1$-lifting principle~\cite{Ran92, KawamataT1Lifting92, Kawamata97, FantechiManetti99}.
For this, let $f \from \frX \to S = \Spec A$ be an infinitesimal \lt deformation of $X$, let $S' = \Spec A' \subset S$ be a closed subscheme and $f' \from \frX' \to S'$ the induced deformation.
Then the $T^1$-lifting property boils down to surjectivity of the canonical map
\begin{equation} \label{648}
\HH1.\frX.\T{\frX/S}. \lto \HH1.\frX'.\T{\frX'/S'}..
\end{equation}

\subsection{Proof of \cref{main 2}}

With notation as above, assume additionally that $\HH2.X.\T X. = 0$.
To apply the $T^1$-lifting principle, it actually suffices to check surjectivity of~\labelcref{648} in the special case $A = A_{n + 1}$ and $A' = A_n$, where $A_n \defn \factor{\C[t]}{(t^{n+1})}$ and the map $A \to A'$ is the natural one~\cite[Thm.~1]{KawamataT1Lifting92}.
For every given $n \in \N_0$, the short exact sequence
\[ 0 \lto (t^{n+1}) \lto A_{n+1} \lto A_n \lto 0 \]
induces by flatness (cf.~the proof of \cref{531})
\[ 0 \lto \O X \lto \O \frX \lto \O{\frX'} \lto 0, \]
which after tensorizing with $\T{\frX/S}$ yields
\begin{equation} \label{T1 ses}
0 \lto \T X \lto \T{\frX/S} \lto \T{\frX'/S'} \lto 0.
\end{equation}
For exactness in the middle and on the right, cf.~the proof of~\cref{gkkp}.
Injectivity on the left is best seen using a local trivialization $\frX \isom X \x S$ and the description of $\T{\frX/S}$ as $p^* \T X$, where $p \from \frX \to X$ is the first projection.
The long exact sequence associated to~\labelcref{T1 ses}, combined with the assumption $\HH2.X.\T X. = 0$, now immediately yields surjectivity of the map~\labelcref{648}, completing the proof. \qed

\subsection{Proof of \cref{main 3}}

We drop the assumption $\HH2.X.\T X. = 0$ and instead assume $X$ to be \kahler with rational singularities and $\can X \isom \O X$.
By \cref{rel can} below, also $\can{\frX/S} \isom \O{\frX}$.
By contraction, we obtain an isomorphism $\T{\frX/S} \isom \T{\frX/S} \tensor \can{\frX/S} = \Omegar{\frX/S}{n-1}$.
Hence the surjectivity of~\labelcref{648} is equivalent to the surjectivity of
\begin{equation} \label{655}
\HH1.\frX.\Omegar{\frX/S}{n-1}. \lto \HH1.\frX'.\Omegar{\frX'/S'}{n-1}..
\end{equation}
We will show that $\RR1.f.\Omegar{\frX/S}{n-1}.$ is locally free and commutes with base change.
In this case,~\labelcref{655} simply is the canonical map
\[ \HH1.\frX.\Omegar{\frX/S}{n-1}. \lto \HH1.\frX.\Omegar{\frX/S}{n-1}. \tensor_A A', \]
which is surjective since $A \to A'$ is.

\addtocounter{thm}{-1}

\begin{lem} \label{rel can}
Let $f \from \frX \to S = \Spec A$ be an infinitesimal \lt deformation of a normal compact \kahler space $X$ with rational singularities and $\can X \isom \O X$.
Then also the relative canonical sheaf $\can{\frX/S}$ is trivial.
\end{lem}

\begin{proof}
Let $g \from \frY \to \frX$ be a simultaneous resolution of $\frX \to S$.
Then the composition $\pi \from \frY \to S$ is a deformation of the compact \kahler manifold $Y$, where $\bar g \from Y \to X$ is the functorial resolution of the central fibre $X$.

By~\cite[Th\'eor\`eme~5.5]{Deligne68}, the sheaf $\pi_* \can{\frY/S}$ is locally free (hence free, as $A$ is local).
Since it also satisfies base change, its rank may be calculated from the central fibre, $\hh0.Y.\can Y. = \hh0.X.\can X. = 1$.
As $\pi_*(-) = f_* g_* (-)$, we thus obtain from \cref{gkkp} that $f_* \can{\frX/S} \isom \O S$.
Pick a global section $\sigma$ of $\can{\frX/S}$ that under the identification
\[ \HH0.\frX.\can{\frX/S}. = \HH0.S.f_* \can{\frX/S}. \isom \HH0.S.\O S. = A \]
corresponds to an element of $A^\x$, that is, to a unit of $A$.
This gives rise to a map $\sigma \from \O{\frX} \lto \can{\frX/S}$, whose restriction to the central fibre is an isomorphism.
By the same argument as in the proof of \cref{gkkp}, we may deduce from this that $\sigma$ itself is an isomorphism.
This ends the proof.
\end{proof}

Recall that we are currently proving \cref{main 3} and hence we also assume that the reflexive Poincar\'e lemma holds on $X$ and that the corresponding spectral sequence~\labelcref{hdr} degenerates at the first page in degree $n$.

\begin{clm}[Relative Poincar\'e lemma on $\frX$ over $S$] \label{rpl}
The complex $\Omegar{\frX/S}\bullet$ is a resolution of the constant sheaf $\underline{A}_\frX$.
\end{clm}

\begin{proof}
The claim is local on $\frX$.
By local triviality of $\frX \to S$, we may therefore assume that $\frX \isom X \x S$.
By assumption, $\Omegar X{\bullet}$ is exact in degree $\ge 1$.
As the projection $p \from X \x S \to X$ is flat, also $p^* \Omegar X\bullet = \Omegar{\frX/S}\bullet$ is exact in degree $\ge 1$.
The kernel of $\O\frX \to \Omegar{\frX/S}1$ consists of those functions that are pulled back from $S$, i.e.~the sheaf $f\inv \O S$.
Since $\O S = \underline{A}_S$, we have $f\inv \O S = \underline{A}_\frX$.
\end{proof}

By flatness of $\Omegar{\frX/S}p$ over $S$ and by~\cite[Ch.~III, Thm.~4.1]{BS76}, there are bounded complexes of finitely generated free $A$-modules $L_p^\bullet$, $1 \le p \le n$, such that
\begin{align} \label{Romanian dudes}
\RR q.f.\big( \Omegar{\frX/S}p \tensor f^* Q \big). & = \mathrm H^q(L_p^\bullet \tensor Q)
\end{align}
for any $q \ge 0$ and any finitely generated $A$-module $Q$.
We will apply these statements to $Q = A/\frm = \C$.
Let $\ell(-)$ denote the length of an $A$-module, and recall that this is additive on short exact sequences.
By~\cite[(3.5)]{Deligne68}, we have the following inequality:
\begin{equation} \label{765}
\ell \big( \mathrm H^q(L_p^\bullet) \big) \le \ell(A) \cdot \dim_\C \mathrm H^q(L_p^\bullet \tensor \C) \qquad \text{for all $q \ge 0$ and $1 \le p \le n$.}
\end{equation}
Furthermore, as soon as we can show that equality holds for $(p, q) = (n - 1, 1)$, it follows from~\cite[(3.5)]{Deligne68} again that $\RR1.f.\Omegar{\frX/S}{n-1}.$ is locally free and commutes with base change, and the proof of \cref{main 3} will be finished.
In other words, using~\labelcref{Romanian dudes}, we have to show that
\begin{equation} \label{1008}
\ell \big( \RR1.f.\Omegar{\frX/S}{n-1}. \big) = \ell(A) \cdot \dim_{\C}\HH1.X.\Omegar X{n-1}..
\end{equation}
To this end, consider the following chain of inequalities:
\begin{align*}
\ell \big( \RR n.f.\Omegar{\frX/S}{\bullet}. \big) & \le \sum_{p+q=n} \ell \big( \RR q.f.\Omegar{\frX/S}p. \big) && \\
& \le \ell(A) \cdot \sum_{p+q=n} \hh q.X.\Omegar Xp. && \text{by~\labelcref{Romanian dudes} and~\labelcref{765}} \\
& \le \ell(A) \cdot \hh n.X.\C. && \text{partial degeneration of~\labelcref{hdr}} \\
& = \ell \big( \HH n.\frX.{\underline{A}_\frX}. \! \big) && \text{universal coefficient theorem.}
\end{align*}
The first inequality comes from the spectral sequence associated to the stupid filtration on the complex $\Omegar{\frX/S}{\bullet}$.
By \cref{rpl} we have $\RR n.f.\Omegar{\frX/S}{\bullet}. \isom \HH n.\frX.{\underline{A}_\frX}.$, and therefore equality holds everywhere in the above chain of inequalities.
In particular, for $(p, q) = (n - 1, 1)$ we get the desired equality~\labelcref{1008}. \qed

\subsection{Proof of \cref{cor 4}}

Let $X$ be a normal compact \kahler space with quotient or toroidal singularities and $K_X \sim_\Q 0$.
Consider the index one cover $X_1 \to X$.
Then $K_{X_1} \sim 0$ and $X_1$ has the same kind of singularities by \cref{index 1} below.
By \cref{deflt quotient}, smoothness of $\Deflt(X_1)$ implies the same property for~$X$.

\begin{prp} \label{index 1}
Let $X$ be a normal complex space such that $K_X \sim_\Q 0$, and let $X_1 \to X$ be the index one cover.
If $X$ has only quotient (resp.~toroidal) singularities, then the same is true for $X_1$.
\end{prp}

\begin{proof}
For quotient singularities, the statement is well-known and can in fact easily be deduced from~\cite[Lemma~2.9]{BochnerCGGN}.
In the toroidal case, we will freely use standard facts and notation from toric geometry as expounded e.g.~in~\cite{CoxLittleSchenck}.
Assume now that $X$ has toroidal singularities, and fix $x \in X$.
Then there is a euclidean open neighborhood $x \in U \subset X$ and an isomorphism $\phi \from U \bij V \subset X'$, where $X'$ is toric and $\phi(x)$ is a torus-fixed point.
By Sumihiro's theorem, we may shrink $X'$ around $\phi(x)$ and assume that $X'$ is affine and $K_{X'} \sim_\Q 0$.
In the remainder of the proof, we will explain the toric construction of the index one cover of $X'$.
We abuse notation and replace $X$ by $X'$.

So let $N, M$ be dual lattices and let $\sigma \subset N_\R$ be a strictly convex rational polyhedral cone with $X = U_{\sigma, N} = \Spec \C[\sigma\dual \cap M]$.
Let $\sigma(1)$ be the set of extremal rays of $\sigma$ and $u_\rho$ the minimal generator of $\rho \cap N$, for any $\rho \in \sigma(1)$.
Since $X$ is in particular \Q-Gorenstein, there exists $m_0 \in M_\Q$ such that $\tw m_0, u_\rho. = 1$ for all $\rho \in \sigma(1)$.
If $r$ is the index of $K_X$, then $rm_0 \in M$ and $r$ is the smallest positive integer with this property.

Set $\wt N \defn m_0\inv(\Z) \subset N$, considering $m_0$ as a map $N \to \Z \cdot \frac1r \subset \Q$.
Then $\wt N$ contains the set $\set{ u_\rho \;|\; \rho \in \sigma(1) }$ and $\factor N{\wt N} \cong \factor {\Z \frac1r}\Z \cong \factor \Z{r\Z}$ is finite cyclic.
In particular $\wt N_\R = N_\R$, $M \subset \wt M$ has finite index and $M_\Q = \wt M_\Q$.
The inclusion $\sigma\dual \cap M \subset \sigma\dual \cap \wt M$ gives rise to a finite cyclic degree $r$ Galois cover
\[ \pi \from X_1 \defn U_{\sigma, \wt N} = \Spec \C[\sigma\dual \cap \wt M] \lto X. \]
Note that when passing from $N$ to $\wt N$, neither $\sigma(1)$ nor the individual $u_\rho$ change.
Therefore, since $m_0 \in \wt M \subset \wt M_\Q$ and we still have $\tw m_0, u_\rho. = 1$, it follows that $K_{X_1}$ is Cartier.
Better still, $K_{X_1} \sim 0$ as $X_1$ is affine.

It remains to show that $\pi$ is quasi-\'etale.
The $T_N$-invariant prime divisors on $X$ are in bijection with the $T_{\wt N}$-invariant prime divisors on $X_1$ because both are in bijection with $\sigma(1)$ by the Orbit--Cone correspondence.
Recall that this correspondence sends $\rho$ to $D_{\rho, X} \defn$ the closure of the orbit of $\gamma_{\rho, X}$, where $\gamma_{\rho, X}$ is the $t \to 0$ limit of the $1$-parameter subgroup $\C^* \to X$ defined by $u_\rho$, and analogously for $D_{\rho, X_1}$.
As $\pi(\gamma_{\rho, X_1}) = \gamma_{\rho_, X}$, it follows that $\pi(D_{\rho, X_1}) = D_{\rho, X}$ set-theoretically.
Even better, since $\pi$ is surjective and the number of invariant divisors on $X$ and on $X_1$ is the same, we must have $\pi\inv(D_{\rho, X}) = D_{\rho, X_1}$ set-theoretically.
In other words, $\pi^*(D_{\rho, X}) = \mu_\rho \cdot D_{\rho, X_1}$ as divisors, and we need to show that all $\mu_\rho = 1$.

Now the divisor of the character $\chi$ corresponding to $rm_0 \in M$ is
\[ \divisor(\chi) = \sum_{\rho \in \sigma(1)} r D_{\rho, X}, \]
and using that $\divisor(\pi^* \chi) = \pi^* \divisor(\chi)$ we get
\[ \sum_{\rho \in \sigma(1)} r D_{\rho, X_1} = \sum_{\rho \in \sigma(1)} r \mu_\rho D_{\rho, X_1}. \]
Comparing coefficients, it follows that $\mu_\rho = 1$ for all $\rho \in \sigma(1)$, as desired.
\end{proof}

\begin{prp} \label{deflt quotient}
Let $Y$ be a normal compact complex space, and let $G$ be a finite group acting effectively on $Y$.
Assume that $\Deflt(Y)$ is smooth.
Then there is a \lt deformation $\pi \from \frY \to \Delta$ of $Y$ over a smooth base $0 \in \Delta$ such that:
\begin{enumerate}
\item The group $G$ acts on $\frY$ in a fibre-preserving way, extending the given action $G \acts Y$.
\item The Kodaira--Spencer map $\kappa(\pi) \from T_0 \Delta \to \HH1.Y.\T Y.$ is an isomorphism onto $\HH1.Y.\T Y.^G \subset \HH1.Y.\T Y.$.
\end{enumerate}
Put $X \defn \factor YG$ and $\phi \from \frX \defn \factor\frY G \to \Delta$.
Then:
\begin{enumerate}
\item The map $\phi$ is a \lt deformation of $X$.
\item If the action of $G$ on $Y$ is free in codimension one, then $\phi$ is the semiuniversal \lt deformation of $X$.
In particular, $\Deflt(X)$ is smooth in this case.
\end{enumerate}
\end{prp}

\begin{proof}
In the case that $Y$ is smooth, this is~\cite[Prop.~6.1, Prop.~6.2]{AlgApprox}.
For general $Y$, the proof still works because we only consider \emph{\lt} deformations of $Y$.
\end{proof}

By the above argument, we may replace $X$ by $X_1$ and assume $K_X \sim 0$.
Our aim is to apply \cref{main 3}, so we first need to know that the reflexive Poincar\'e lemma holds on $X$.
For quotient singularities, this is~\cite[Cor.~1.9]{Steenbrink77}.
In the toroidal case, we use~\cite[Lemma~2.3]{Danilov91} instead.
Note that these references work in an algebraic context, i.e.~$X$ is assumed to be an algebraic variety.
This, however, is not a problem since the claim and its proof are analytically local.

We also need to know that~\labelcref{hdr} degenerates at $E_1$ in degree $n = \dim X$.
In our situation, this degeneration even holds in any degree: in the orbifold case, apply~\cite[Thm.~1.12]{Steenbrink77}, and for toroidal spaces, use~\cite[Thm.~3.4.a)]{Danilov91}.
Again, the results are stated only for algebraic varieties.
But a close inspection of the proofs reveals that this assumption is, in both cases, only invoked once in a substantial way: $X$ has a resolution of singularities on which~\labelcref{hdr} degenerates.
Since this still holds true for compact \kahler spaces, degeneration remains valid in our setting.

It is well-known that quotient singularities as well as \Q-Gorenstein toroidal singularities are klt~\cite[Prop.~5.13]{KM98}, \cite[Cor.~11.4.25(b)]{CoxLittleSchenck}.
In particular, they are rational~\cite[Thm.~5.22]{KM98}.
\cref{cor 4} is thus a consequence of \cref{main 3}. \qed

\subsection{Proof of \cref{cor 5}}

This immediately follows from Theorems~\labelcref{main unobstructed},~\labelcref{main 2} and \cref{cor 4}. \qed

\providecommand{\bysame}{\leavevmode\hbox to3em{\hrulefill}\thinspace}
\providecommand{\MR}{\relax\ifhmode\unskip\space\fi MR}
\providecommand{\MRhref}[2]{%
  \href{http://www.ams.org/mathscinet-getitem?mr=#1}{#2}
}
\providecommand{\href}[2]{#2}


\begin{thebibliography}{CGGN20}

\bibitem[BGL20]{BakkerGuenanciaLehn20}
{\sc B.~Bakker, H.~Guenancia, and C.~Lehn}: \emph{{Algebraic approximation and
  the decomposition theorem for K{\"a}hler Calabi--Yau varieties}},
  \href{http://arxiv.org/abs/2012.00441}{arXiv:2012.00441 [math.AG]}, December
  2020.

\bibitem[BL19]{BakkerLehn18}
{\sc B.~Bakker and C.~Lehn}: \emph{{The global moduli theory of symplectic
  varieties}}, \href{http://arxiv.org/abs/1812.09748}{arXiv:1812.09748
  [math.AG]}, version 2, September 2019.

\bibitem[BS76]{BS76}
{\sc C.~B{\u a}nic{\u a} and O.~St{\u a}n{\u a}{\c s}il{\u a}}:
  \emph{{Algebraic Methods in the Global Theory of Complex Spaces}}, John Wiley
  \& Sons, 1976.

\bibitem[Bin83]{Bin83}
{\sc J.~Bingener}: \emph{On deformations of {K}\"ahler spaces. {I}}, Math. Z.
  \textbf{182} (1983), no.~4, 505--535.

\bibitem[Bog78]{Bogomolov78}
{\sc F.~A. Bogomolov}: \emph{Hamiltonian {K}\"{a}hlerian manifolds},
  Dokl.~Akad.~Nauk SSSR \textbf{243} (1978), no.~5, 1101--1104.

\bibitem[Cam04]{Cam04FanoConference}
{\sc F.~Campana}: \emph{Orbifoldes {\`a} premi{\`e}re classe de {C}hern nulle},
  The {F}ano {C}onference, Univ. Torino, Turin, 2004, pp.~339--351.

\bibitem[CGGN20]{BochnerCGGN}
{\sc B.~Claudon, P.~Graf, H.~Guenancia, and P.~Naumann}: \emph{{K{\"a}hler
  spaces with zero first Chern class: Bochner principle, fundamental groups,
  and the Kodaira problem}},
  \href{http://arxiv.org/abs/2008.13008}{arXiv:2008.13008 [math.AG]}, August
  2020.

\bibitem[CHL19]{ClaudonHoeringLinKahlerGroups}
{\sc B.~Claudon, A.~H{\"o}ring, and H.-Y. Lin}: \emph{{The fundamental group of
  compact K{\"a}hler threefolds}}, Geometry \& Topology \textbf{23} (2019),
  3233--3271.

\bibitem[CLS11]{CoxLittleSchenck}
{\sc D.~A. Cox, J.~B. Little, and H.~K. Schenck}: \emph{Toric varieties},
  Graduate Studies in Mathematics, vol. 124, American Mathematical Society,
  Providence, RI, 2011.

\bibitem[Dan91]{Danilov91}
{\sc V.~I. Danilov}: \emph{De {R}ham complex on toroidal variety}, Algebraic
  geometry ({C}hicago, {IL}, 1989), Lecture Notes in Math., vol. 1479,
  Springer, Berlin, 1991, pp.~26--38.

\bibitem[Del68]{Deligne68}
{\sc P.~Deligne}: \emph{Th\'{e}or\`eme de {L}efschetz et crit\`eres de
  d\'{e}g\'{e}n\'{e}rescence de suites spectrales}, Publ.~Math.~de l'I.H.\'E.S.
  (1968), no.~35, 259--278.

\bibitem[Elk78]{Elkik78}
{\sc R.~Elkik}: \emph{Singularit\'{e}s rationnelles et d\'{e}formations},
  Invent.~Math. \textbf{47} (1978), no.~2, 139--147.

\bibitem[FM99]{FantechiManetti99}
{\sc B.~Fantechi and M.~Manetti}: \emph{On the {$T^1$}-lifting theorem},
  J.~Algebraic Geom. \textbf{8} (1999), no.~1, 31--39.

\bibitem[FK87]{FK87}
{\sc H.~Flenner and S.~Kosarew}: \emph{{On Locally Trivial Deformations}},
  Publ.~RIMS Kyoto \textbf{23} (1987), 627--665.

\bibitem[Gra18]{AlgApprox}
{\sc P.~Graf}: \emph{{Algebraic approximation of K{\"a}hler threefolds of
  Kodaira dimension zero}}, Math.~Annalen \textbf{371} (2018), 487--516.

\bibitem[Gra21]{KodairaFour}
{\sc P.~Graf}: \emph{{A decomposition theorem for singular K{\"a}hler spaces
  with trivial first Chern class of dimension at most four}},
  \href{http://arxiv.org/abs/2101.06764}{arXiv:2101.06764 [math.AG]}, January
  2021.

\bibitem[GK20]{TorusQuotients}
{\sc P.~Graf and T.~Kirschner}: \emph{{Finite quotients of three-dimensional
  complex tori}}, Ann.~Inst.~Fourier (Grenoble) \textbf{70} (2020), no.~2,
  881--914.

\bibitem[GK14]{GK13}
{\sc P.~Graf and S.~J. Kov{\'a}cs}: \emph{{An optimal extension theorem for
  $1$-forms and the {L}ipman-{Z}ariski Conjecture}}, Documenta Math.
  \textbf{19} (2014), 815--830.

\bibitem[GS20]{MoriKodaira}
{\sc P.~Graf and M.~Schwald}: \emph{{On the Kodaira problem for uniruled
  K{\"a}hler spaces}}, Ark.~Mat. \textbf{58} (2020), no.~2, 267--284.

\bibitem[GK64]{GK64}
{\sc H.~Grauert and H.~Kerner}: \emph{{Deformationen von Singularit{\"a}ten
  komplexer R{\"a}ume}}, Math.~Annalen \textbf{153} (1964), 236--260.

\bibitem[Gro97]{Gro97}
{\sc M.~Gross}: \emph{{The deformation space of Calabi--Yau $n$-folds with
  canonical singularities can be obstructed}}, Mirror Symmetry, II, AMS/IP
  Studies in Advanced Mathematics, vol.~1, Amer.~Math.~Soc., 1997,
  pp.~401--411.

\bibitem[Huy05]{Huy05}
{\sc D.~Huybrechts}: \emph{Complex geometry}, Universitext, Springer-Verlag,
  Berlin, 2005.

\bibitem[J{\"o}r14]{Joerder}
{\sc C.~J{\"o}rder}: \emph{{On the Poincar{\'e} Lemma for reflexive
  differential forms}}, PhD thesis, Albert-Ludwigs-Universit{\"a}t Freiburg im
  Breisgau. Available at \url{https://freidok.uni-freiburg.de/data/9438}, 2014.

\bibitem[Kaw92]{KawamataT1Lifting92}
{\sc Y.~Kawamata}: \emph{Unobstructed deformations. {A} remark on a paper of
  {Z}.~{R}an}, J.~Algebraic Geom. \textbf{1} (1992), no.~2, 183--190.

\bibitem[Kaw97]{Kawamata97}
{\sc Y.~Kawamata}: \emph{Erratum on: ``{U}nobstructed deformations. {A} remark
  on a paper of {Z}. {R}an''}, J.~Algebraic Geom. \textbf{6} (1997), no.~4,
  803--804.

\bibitem[KS19]{KebekusSchnell18}
{\sc S.~Kebekus and C.~Schnell}: \emph{Extending holomorphic forms from the
  regular locus of a complex space to a resolution of singularities},
  \href{http://arxiv.org/abs/1811.03644}{arXiv:1811.03644 [math.AG]},
  version~3, February 2019.

\bibitem[KNS58]{KodairaNirenbergSpencer58}
{\sc K.~Kodaira, L.~Nirenberg, and D.~C. Spencer}: \emph{On the existence of
  deformations of complex analytic structures}, Ann.~of Math.~(2) \textbf{68}
  (1958), 450--459.

\bibitem[Kod63]{Kod63}
{\sc K.~Kodaira}: \emph{{On Compact Analytic Surfaces, III}}, Ann.~Math.
  \textbf{78} (1963), no.~1, 1--40.

\bibitem[Kol07]{Kol07}
{\sc J.~Koll{\'a}r}: \emph{{Lectures on resolution of singularities}}, Annals
  of Mathematics Studies, vol. 166, Princeton University Press, Princeton, NJ,
  2007.

\bibitem[KM98]{KM98}
{\sc J.~Koll{\'a}r and S.~Mori}: \emph{Birational geometry of algebraic
  varieties}, Cambridge Tracts in Mathematics, vol. 134, Cambridge University
  Press, Cambridge, 1998.

\bibitem[Lin18a]{Lin17b}
{\sc H.-Y. Lin}: \emph{{Algebraic approximations of compact K{\"a}hler
  threefolds}}, \href{http://arxiv.org/abs/1710.01083}{arXiv:1710.01083
  [math.AG]}, version 2, November 2018.

\bibitem[Lin18b]{Lin16}
{\sc H.-Y. Lin}: \emph{{Algebraic approximations of fibrations in abelian
  varieties over a curve}},
  \href{http://arxiv.org/abs/1612.09271}{arXiv:1612.09271 [math.AG]}, version
  2, November 2018.

\bibitem[Mas99]{Mas99}
{\sc H.~Maschke}: \emph{{Beweis des Satzes, dass diejenigen endlichen linearen
  Substitutionsgruppen, in welchen einige durchgehends verschwindende
  Coefficienten auftreten, intransitiv sind}}, Math.~Ann. \textbf{52} (1899),
  no.~2--3, 363--368.

\bibitem[Nam02]{Nam02}
{\sc Y.~Namikawa}: \emph{{Projectivity criterion of Moishezon spaces and
  density of projective symplectic varieties}}, Int.~J.~Math. \textbf{13}
  (2002), no.~2, 125--135.

\bibitem[Ran92]{Ran92}
{\sc Z.~Ran}: \emph{Deformations of manifolds with torsion or negative
  canonical bundle}, J.~Alg.~Geom. \textbf{1} (1992), no.~2, 279--291.

\bibitem[Ser06]{Ser06}
{\sc E.~Sernesi}: \emph{{Deformations of Algebraic Schemes}}, {Grundlehren der
  mathe\-matischen Wissenschaften}, vol. 334, Springer-Verlag, 2006.

\bibitem[Ste77]{Steenbrink77}
{\sc J.~H.~M. Steenbrink}: \emph{Mixed {H}odge structure on the vanishing
  cohomology}, Real and complex singularities, Proc.~Ninth Nordic Summer
  School/NAVF Sympos.~Math., Oslo, 1976, Sijthoff and Noordhoff, Alphen aan den
  Rijn, 1977, pp.~525--563.

\bibitem[Tak85]{Takegoshi85}
{\sc K.~Takegoshi}: \emph{Relative vanishing theorems in analytic spaces}, Duke
  Math.~J. \textbf{52} (1985), no.~1, 273--279.

\bibitem[Tia87]{Tian87}
{\sc G.~Tian}: \emph{Smoothness of the universal deformation space of compact
  {C}alabi-{Y}au manifolds and its {P}etersson-{W}eil metric}, Mathematical
  aspects of string theory ({S}an {D}iego, {C}alif., 1986), Adv.~Ser.
  Math.~Phys., vol.~1, World Sci.~Publishing, Singapore, 1987, pp.~629--646.

\bibitem[Tod89]{Todorov89}
{\sc A.~N. Todorov}: \emph{The {W}eil-{P}etersson geometry of the moduli space
  of {${\rm SU}(n\geq 3)$} ({C}alabi-{Y}au) manifolds. {I}}, Comm.~Math.~Phys.
  \textbf{126} (1989), no.~2, 325--346.

\bibitem[Var89]{Var89}
{\sc J.~Varouchas}: \emph{K\"ahler spaces and proper open morphisms},
  Math.~Ann. \textbf{283} (1989), no.~1, 13--52.

\bibitem[Voi02]{Voi02}
{\sc C.~Voisin}: \emph{{Hodge Theory and Complex Algebraic Geometry I}},
  Cambridge Studies in Advanced Mathematics, vol.~76, Cambridge University
  Press, 2002.

\bibitem[Voi03]{Voi03}
{\sc C.~Voisin}: \emph{{Hodge Theory and Complex Algebraic Geometry II}},
  Cambridge Studies in Advanced Mathematics, vol.~77, Cambridge University
  Press, 2003.

\bibitem[Voi04]{Voi04}
{\sc C.~Voisin}: \emph{{On the homotopy types of compact K{\"a}hler and complex
  projective manifolds}}, Invent.~Math. \textbf{157} (2004), 329--343.

\bibitem[Voi06]{Voi06}
{\sc C.~Voisin}: \emph{{On the homotopy types of K{\"a}hler manifolds and the
  birational Kodaira problem}}, J.~Diff.~Geom. \textbf{72} (2006), 43--71.

\end{thebibliography}
\end{document}